\documentclass[fleqn,11pt]{article}
\usepackage{amsmath,amssymb,amsfonts,amsthm,graphicx}
\usepackage{color}

\newtheorem{Theorem}{Theorem}

\newtheorem{Proposition}[Theorem]{Proposition}
\newtheorem{Definition}[Theorem]{Definition}
\newtheorem{Corollary}[Theorem]{Corollary}

\newtheorem{Remark}[Theorem]{Remark}

\begin{document}

\title{\textbf{Some inequalities for log-convex functions of selfadjoint operators on quaternionic\\ Hilbert spaces}}
\author{ M. Fashandi\footnote{{\small{E-mail:}  fashandi@um.ac.ir (M. Fashandi)}} \\
\footnotesize{  Faculty of Mathematical Sciences, Ferdowsi University of Mashhad,}\\  
  \footnotesize{P.O. Box 1159,  Mashhad, 91775\ Iran }
  }

\date{}
\maketitle

\begin{abstract}
In this paper,  some Jensen's type  inequalities between quaternionic bounded selfadjoint operators on  quaternionic Hilbert spaces are proved, using a log-convex function. 
Also, by applying a specific log-convex function, some particular cases of operator inequalities are obtained.
\end{abstract}

\noindent {\bf Keywords:} Quaternionic Hilbert spaces, Continuous functional calculus, Log-convex function, Operator inequality.\\
\noindent {\it 2010 Mathematics Subject Classification MSC: 47A63, 47A99}  \\
\maketitle

\section{Introduction}
When the exact value of something is not known we must resort to inequalities. Inequalities have been around since the Mesopotamian and Egyptian civilizations, where they were used to figure out how to measure land and build things. Triangle inequality and arithmetic-geometric mean inequality for two numbers, are the initial inequalities that ancient Greeks knew (see \cite{Fink} for comprehensive information about the history of inequalities in Mathematics). Up to now, many inequalities have been proved  among them, some  are ``named inequality" that are ``attributable to someone who stated it or proved it" and  ``someone must judge that it is important enough to be referred to by name and others must find attribution important" (\cite{Fink}). In his retiring presidential (London Mathematical Society) address, G. H. Hardy on November 8, 1928,  said  ``all analysts spend half their time hunting through the literature for inequalities which they want to use and cannot prove" (\cite{Fink}). The importance of  inequalities is not limited to obtaining bounds, many definitions in mathematics are based on inequalities such as the definition of limit or convex functions, etc. With the progress of mathematics, not only new inequalities were proved, but some old inequalities have been refined and some of them found new applications in science and engineering. In addition, the extension of inequalities to other spaces with new tools has led to growth in various mathematical disciplines. One of the subjects of high interest is operator inequality, especially for the operators defined on Hilbert spaces (see for instance \cite{Drag2}, \cite{Drag},  \cite{Huy}, \cite{Kit}, \cite{Mond} and \cite{Furuta1}) and another path was found to extend operator inequalities to the spaces similar to Hilbert spaces, such as Hilbert $C^*$-modules, direct sum of Hilbert spaces, etc., or different kind of operators (see \cite{Fuji} and \cite{Drag4}). Quaternionic Hilbert spaces have  recently attracted the attention of physicists and mathematical physicists, as they are  more suitable for dealing with quantum systems (for more information, see \cite{Ghi1}). Getting to know quaternionic Hilbert spaces, made the authors think about operator inequalities on such spaces and recently some papers have been published in this direction (see \cite{Kont} and \cite{Adv.}). In this paper, we extend some of the operator inequalities related to selfadjoint operators on complex Hilbert spaces which have been proved in \cite{Drag}, to the ones on right quaternionic Hilbert spaces. The results of this study are based on the inequalities obtained  in  \cite{Adv.}.

\section{Notations and Preliminaries}

Through this paper,  $\mathbb{H}$ will stand for the skew field of quaternions, whose elements are in the form $q= x_0+x_1i+ x_2j+x_3k$, where $x_0,x_1, x_2$ and $x_3$ are real numbers and $i, j,$ and $k$ are called imaginary units and obey the following multiplication rules:
\begin{equation*}
i^2=j^2=k^2=-1,\;\; ij=-ji=k,\;\; jk=-kj=i,\;\; {\rm\mbox and}\;\; ki=-ik=j.
\end{equation*}
Also, $\mathsf{H}$ denotes a right quaternionic Hilbert space which is a linear vector space over the skew field of  quaternions under right scalar multiplication and an inner product 
$\langle.,.\rangle:\mathsf{H}\times \mathsf{H}\longrightarrow \mathbb{H}$
with the following properties \begin{enumerate}
\item[(i)] $\overline{\langle u, v\rangle}=\langle v, u\rangle$,
\item[(ii)]$\langle u, u\rangle>0$ unless $u=0$,
\item[(iii)] $\langle u,vp+wq\rangle=\langle u, v\rangle p+\langle u, w\rangle q$,
\end{enumerate}
for every $u, v, w\in \mathsf{H}$ and $p, q\in \mathbb{H}$ and the quaternionic norm  $\Vert u \Vert= \sqrt{\langle u, u\rangle}$
that is assumed to produce a complete metric space (see  Proposition 2.2 in  \cite{Ghi1}). 
The notation $\mathfrak{B}(\mathsf{H})$ shows the set all bounded right linear operators on $\mathsf{H}$ that means 
 $$T(up+v)=(Tu)p+Tv,$$
for all $u, v \in \mathsf{H}$ and $p\in \mathbb{H}$. By  Proposition 2.11 in \cite{Ghi1}, $\mathfrak{B}(\mathsf{H})$  is a complete normed space with the norm defined by 
\begin{equation}\label{norm1}
\Vert T\Vert= \sup \left\{\dfrac{\Vert Tu\Vert}{\Vert u\Vert}, 0\neq u\in\mathsf{H}\right\}.
\end{equation}
Adjoint of $T\in \mathfrak{B}(\mathsf{H})$, selfadjoint, positive, normal and unitary right linear operators in $\mathfrak{B}(\mathsf{H})$ are defined similar to the complex case (see Definition 2.12 in \cite{Ghi1}).
In section 4 in \cite{Ghi1}, one can see the extention of the definitions of spectrum and resolvent in $\mathsf{H}$, as follows.
\begin{Definition}\label{spectrum}\cite{Ghi1}
For  $T\in \mathfrak{B}(\mathsf{H})$ and $q\in \mathbb{H}$, the associated operator $\Delta_{q}(T)$ is defined by:
$$
\Delta_q(T):=T^2- T(q+\bar{q})+I \vert q\vert^2.
$$
The spherical resolvent set of $T\in \mathfrak{B}(\mathsf{H})$ is the set
$$
\rho_{S}(T):=\{q\in\mathbb{H}:\,\,\, \Delta_q(T)^{-1}\in \mathfrak{B}(\mathsf{H})\}.
$$
\end{Definition}
\noindent The \textit{spherical spectrum} $\sigma_{S}(T)$ of $T$ is defined as the complement of $\rho_{S}(T)$ in $\mathbb{H}$.  A partition for $\sigma_{S}(T)$ was introduced in \cite{Ghi1}, as follows:
\begin{itemize}
\item[(i)] The \textit{spherical point spectrum} of $T$:
$$\sigma_{pS}(T)=\{q\in\mathbb{H}; \ker(\Delta_{q}(T))\neq\{0\}\}.$$
\item[(ii)] The\textit{ spherical residual spectrum} of $T$:
$$\sigma_{rS}(T)=\{q\in\mathbb{H}; \ker(\Delta_{q}(T))=\{0\}, \overline{Ran (\Delta_{q}(T))}\neq\mathsf{H}\}.$$
\item[(iii)] The spherical continuous spectrum of $T$:
$$\sigma_{cS}(T)=\{q\in\mathbb{H}; \ker(\Delta_{q}(T))=\{0\}, \overline{Ran (\Delta_{q}(T))}=\mathsf{H}, 
\Delta_{q}(T)^{-1}\notin \mathfrak{B}(\mathsf{H})\}.$$
\end{itemize}
The \textit{spherical spectral radius} of $T$, denoted by $r_{S}(T)$,  is defined by:
\begin{equation}\label{radius}
r_{S}(T)=\sup\{\vert q\vert \in \mathbb{R}^{+}; q\in \sigma_{S}(T)\}.
\end{equation}
An \textit{eigenvector} of $T$ with \textit{right} \textit{eigenvalue} $q$ is an element $u\in \mathsf{H}-\{0\}$, for which $Tu=uq$.
According to Proposition 4.5 in \cite{Ghi1}, $q$ is a right eigenvalue of $T$ if and only if it  belongs to the spherical  point spectrum $\sigma_{pS}(T)$. 

\noindent To start studying operator inequalities we must be familiar with continuous functional calculus. Fortunately, Ghiloni et al. in \cite{Ghi1} have paved the way to this matter in the right quaternionic Hilbert space and we borrow some results from \cite{Ghi1} exactly to use them throughout this paper which are stated in the next remark.

\begin{Remark}\label{remark}
For $T\in\mathfrak{B}(\mathsf{H})$ the following assertions hold.
\begin{itemize}
\item[(a)] Let $q\in \mathbb{H}$ with $|q|> \Vert T\Vert$ and, for each $n\in \mathbb{N}$, let $a_n$ be the real number defined by setting $a_n:= \vert q \vert^{-2n-2}\sum_{h=0}^n q^h\, (\overline{q})^{n-h}$. Then, the series $\sum_{n\in \mathbb{N}}T^n a_n$ converges absolutely in $\mathfrak{B}(\mathsf{H})$ (with respect to the operator norm $\Vert .\Vert$) to $\Delta_q (T)^{-1}$. In particular, $r_S(T)\leq \Vert T\Vert$ (see Theorem 4.3.(a) in \cite{Ghi1}).
\item[(b)] If $T$ is normal, then $r_S(T)=\Vert T\Vert$ (see Theorem 4.3.(e) in \cite{Ghi1}).
\item[(c)] $\sigma_S (T)$ is a non--empty compact subset of $\mathbb{H}$ (see Theorem 4.3.(b) in \cite{Ghi1}).
\item[(d)] If $T$ is selfadjoint, then $\sigma_S(T)\subset \mathbb{R}$ and $\sigma_{rS}(T)$ is empty (see Theorem 4.8.(b) in \cite{Ghi1}).
\item[(e)] If $T$ is selfadjoint, then there exists a unique continuous homomorphism 
\begin{eqnarray*}
\Phi_T: \mathcal{C}(\sigma_S(T), \mathbb{R})\ni f\mapsto f(T)\in \mathfrak{B}(\mathsf{H})
\end{eqnarray*}
of real Banach unital algebras such that:
\begin{itemize}
\item[(i)] The operator $f(T)$ is selfadjoint for every $f\in \mathcal{C}(\sigma_S(T), \mathbb{R})$.
\item[(ii)] $\Phi_T$ is isometric; that is $\Vert f(T)\Vert=\Vert f\Vert_\infty$ for every $f\in \mathcal{C}(\sigma_S(T), \mathbb{R})$.
\item[(iii)] $\Phi_T$ is positive; that is, $f(T)\geq 0$ if $f \in  \mathcal{C}(\sigma_S(T), \mathbb{R})$ and $f(t)\geq 0$ for every $t\in \sigma_S(T)$.
\end{itemize}
By $\mathcal{C}(\sigma_S(T), \mathbb{R})$, we mean the commutative real Banach unital algebra  of continuous real-valued functions defined on $\sigma_S(T)$ (see Theorem 5.5 in \cite{Ghi1}). 
\end{itemize}
\end{Remark}

\begin{Remark} Note that Remark \ref{remark}(iii) implies that  for a selfadjoint operator $T\in\mathfrak{B}(\mathsf{H})$, if $f, g \in  \mathcal{C}(\sigma_S(T), \mathbb{R})$ and $f(t)\geq g(t)$ for every $t\in \sigma_S(T)$, then $f(T)\geq g(T)$.\end{Remark}
\noindent  For  two operators $S$ and $T \in \mathfrak{B}(\mathsf{H})$, we write $S \leq T$ if $T-S$ is a positive operator, i.e. $\langle Sx, x\rangle \leq \langle Tx, x\rangle$ for every $x\in \mathsf{H}$. In particular,  for two real numbers $m<M$, by $mI\leq T\leq MI$,  we mean
$
m\langle x, x\rangle\leq \langle Tx, x\rangle\leq M\langle x, x\rangle
$
for all $x\in \mathsf{H}$.

\noindent In this paper, we focus on selfadjoint operators on the right quaternionic Hilbert space $\mathsf{H}$ and since the continuous functional calculus is closely related to $\sigma_S(T)$, it is necessary to know it. In Theorem 3 in \cite{Adv.}, enough information about $\sigma_S(T)$ has been obtained.
\begin{Theorem}
\label{compact interval}
For $S, T\in\mathfrak{B}(\mathsf{H})$,
\begin{itemize}
\item[(i)]
If $T$ is a selfadjoint  operator and
\begin{equation*}
\label{infsup}
m_T=\inf \{\langle Tx,x\rangle\,:  \Vert x\Vert=1,\,x\in \mathsf{H}\}, \, \mbox{and}\,\, M_T=\sup \{\langle Tx,x\rangle\, : \Vert x\Vert=1,\,x\in \mathsf{H}\},
\end{equation*}
then $\sigma_S(T)\subset [m_T, M_T]$, $m_T=\min \sigma_S(T)$ and  $M_T=\max \sigma_S(T)$(see Theorem 3 in\cite{Adv.}).
\item[(ii)]
$\sigma_S(ST)\cup \{0 \}=\sigma_S(TS)\cup \{0 \}$ and $r_S(ST)=r_S(TS)$ (see Theorem 5  in \cite{Adv.}).
\item[(iii)] if  $ST=TS$ then $\sigma_S(S+T)\subset \sigma_S(T)+\sigma_S(S)$ (see Proposition 5.1. in \cite{Berb}).
\end{itemize}
\end{Theorem}
Finally,  we need  the quaternionic Mond-Pe\v{c}ari\'{c},  Lah-Ribari\v{c} and Holder-McCarthy inequalities from \cite{Adv.} to prove the results of this paper  (see Theorems 6  and 13 in \cite{Adv.}).
\begin{Theorem}\label{Mond}\cite{Adv.} 
Let $T\in \mathfrak{B}(\mathsf{H})$ be a selfadjoint operator with $\sigma_S(T)\subset [m, M]$ for two real numbers $m<M$. If  $f:[m, M]\longrightarrow \mathbb{R}$ is a continuous convex function, then for each unit vector $x\in\mathsf{H}$,
\begin{itemize}
\item[(i)] (Quaternionic Mond-Pe\v{c}ari\'{c} inequality) $f(\langle Tx, x\rangle)\leq \langle f(T)x, x\rangle,$
\item[(ii)]  (Quaternionic Lah-Ribari\v{c} inequality) $\langle f(T)x, x\rangle \leq \dfrac{M-\langle Tx, x\rangle}{M-m}f(m)+\dfrac{\langle Tx, x\rangle -m}{M-m}f(M).$
\end{itemize}
\end{Theorem}
\begin{Theorem}\label{Holder-McCarthy}\cite{Adv.} (Quatenionic Holder-McCarthy inequality)
If $T\in\mathfrak{B}(\mathsf{H})$ is a positive operator, then
\begin{itemize}
\item[(i)] $\langle T^r x, x\rangle \geq \langle Tx, x\rangle ^r$ for all $r>1$ and any unit vector $x\in \mathsf{H},$
\item[(ii)] $\langle T^r x, x\rangle \leq \langle Tx, x\rangle ^r$ for all $0<r<1$ and any unit vector $x\in \mathsf{H},$
\item[(iii)] If $T$ is invertible, then $\langle T^r x, x\rangle \geq \langle Tx, x\rangle ^r$ for all $r<0$ and any unit vector $x\in \mathsf{H}$.
\end{itemize}
\end{Theorem}


\section{Main results}

A log-convex function $f$ is a positive real function for which $\ln f$ is a convex function.  In \cite{Adv.} some quaternionic operator inequalities were obtained using a convex function. In this section, some refinements are derived from applying a log-convex function for the inequalities in \cite{Adv.} and we emphasize that we follow \cite{Drag} and extend its results to the quaternionic case. Although the results seem to come easily, a careful reader has found that a long way has been traveled to acquire the necessary tools. 

\noindent The first result of this paper is obtained by substituting a positive continuous log-convex function for the continuous convex function $f$  in the quaternionic Mond-Pe\v{c}ari\'{c} inequality.
\begin{Theorem}
\label{Mondlog}
Let $T\in \mathfrak{B}(\mathsf{H})$ be a selfadjoint operator with $\sigma_S(T)\subset [m, M]$ for two real numbers $m<M$. If  $f:[m, M]\longrightarrow (0, +\infty)$ is a continuous log-convex function, then for each unit vector $x\in\mathsf{H}$,
\begin{equation*}
f(\langle Tx, x\rangle)\leq \exp \langle \ln f(T)x, x\rangle \leq \langle f(T)x, x\rangle.
\end{equation*}
\end{Theorem}
\begin{proof}
Remembering that a log-convex function is a function $f:[m, M]\longrightarrow (0, +\infty)$ for which $\ln f$ is convex, then, Theorem \ref{Mond}. (i) provides the result. One may also see the proof of Theorem 2 in \cite{Drag}.
\end{proof}
Combining Theorem 7 in \cite{Adv.} and Theorem \ref{Mondlog}, the following result is obtained.
\begin{Corollary}\label{Mondlog1}
If $f:[m, M]\longrightarrow \mathbb{R}$ is a continuous log-convex function, and $T_j\in \mathfrak{B}(\mathsf{H})$ are selfadjoint operators with $\sigma_S(T_j)\subset [m, M], j\in\{1,...,n\}$  and $x_j\in\mathsf{H}$  with $\sum_{j=1}^n \Vert x_j\Vert ^2=1$, then
\begin{equation*}
f\left( \sum_{j=1}^n \langle T_j x_j, x_j \rangle \right) \leq \exp \Biggl \langle \sum_{j=1}^n \ln f(T_j)x_j, x_j \Biggr \rangle \leq \Biggl \langle \sum_{j=1}^n  f(T_j)x_j, x_j \Biggr\rangle.
\end{equation*}
\end{Corollary}
\noindent Now, using quaternionic Mond-Pe\v{c}ari\'{c} inequality for the log-convex function $f(t)=t^{r}, r<0$, implies a refined version of Theorem \ref{Holder-McCarthy} (iii).
\begin{Proposition}
If $T\in\mathfrak{B}(\mathsf{H})$ is a positive and invertible operator, then for all $r<0$ and unit vector $x\in \mathsf{H}$
\begin{equation*}
\langle Tx, x \rangle ^{r} \leq \exp \langle \ln (T^{r})x, x\rangle \leq  \langle T^{r}x, x\rangle.
\end{equation*}
\end{Proposition}
\noindent Quaternionic version of Theorem 5 in \cite{Drag} which improves quaternionic Lah-Ribari\v{c} inequality, is proved similarly.
\begin{Theorem}\label{Lah}
Let $T\in \mathfrak{B}(\mathsf{H})$ be a selfadjoint operator with $\sigma_S(T)\subset [m, M]$ for two real numbers $m<M$. If $f:[m, M]\longrightarrow (0, +\infty)$ is  log-convex, then for each unit vector $x\in\mathsf{H}$
\begin{itemize}
\item[(i)]$\langle f(T)x, x \rangle \leq \Biggl\langle \Biggl [ [f(m)]^{\frac{MI-T}{M-m}} [f(M)]^{\frac{T-m I}{M-m}} \Bigg ]x , x\Biggr\rangle\leq \dfrac{M-\langle Tx, x\rangle}{M-m}f(m)+\dfrac{\langle Tx, x\rangle -m}{M-m}f(M),$
\item[(ii)]$f(\langle Tx, x \rangle) \leq [f(m)]^{\frac{M-\langle Tx,x \rangle}{M-m}} [f(M)]^{\frac{\langle Tx,x \rangle -m}{M-m}}\leq \Biggl\langle \Biggl [ [f(m)]^{\frac{MI-T}{M-m}} [f(M)]^{\frac{T-m I}{M-m}} \Bigg ]x , x\Biggr\rangle.$
\end{itemize}
\end{Theorem}

\noindent It is easy to show that Corollaries 3 and 4 of \cite{Drag} are valid in quaternionic case. The following theorem is the quaternionic version of Proposition 2 of \cite{Drag}, with the same argument. 
\begin{Theorem}
Let $T\in \mathfrak{B}(\mathsf{H})$ be a positive selfadjoint operator with $\sigma_S(T)\subset [m, M]$ for two real numbers $0<m<M$. If $T$ is invertible, then for each unit vector $x\in\mathsf{H}$ and $r<0$
\begin{itemize}
\item[(i)]$\langle T^r x, x \rangle \leq \Biggl\langle \Biggl [ m^{\frac{MI-T}{M-m}} M^{\frac{T-m I}{M-m}} \Bigg ]^rx , x\Biggr\rangle\leq \dfrac{M-\langle Tx, x\rangle}{M-m}m^r+\dfrac{\langle Tx, x\rangle -m}{M-m}M^r,$
\item[(ii)]$\langle Tx, x \rangle^r \leq \Biggl [ f(m)^{\frac{M-\langle Tx,x \rangle}{M-m}} f(M)^{\frac{\langle Tx,x \rangle -m}{M-m}} \Biggr ]^r \leq \Biggl\langle \Biggl [ m^{\frac{MI-T}{M-m}} M^{\frac{T-m I}{M-m}} \Bigg ]^r x , x\Biggr\rangle.$
\end{itemize}
\end{Theorem}
\noindent Theorem 2.1 of \cite{Drag2} and Proposition 3 of \cite{Drag} can be deduced with the same procedure in the quaternionic setting, stated in the following theorem.
\begin{Theorem}\label{inequalty'}
Let $J$ be an interval and $f:J\rightarrow \mathbb{R}$ be a differentiable function on $\mathring{J}$ (the interior of $J$), whose derivative $f'$ is continuous on $\mathring{J}$. If $T\in \mathfrak{B}(\mathsf{H})$ is a selfadjoint operator with $\sigma_S(T)\subseteq [m, M] \subset \mathring{J}$ for two real numbers $m<M$, then for every unit vector $x\in\mathsf{H}$
\begin{itemize}
\item[(i)] if $f$ is a convex function on $\mathring{J}$
\begin{equation*} 0\leq \langle f(T)x, x \rangle -f (\langle Tx, x\rangle) \leq \langle f'(T)Tx, x \rangle -\langle Tx, x\rangle \langle f'(T)x, x \rangle, 
\end{equation*}
\item[(ii)] if $f$ is a positive log-convex function on $\mathring{J}$
\begin{equation*}
1\leq \dfrac{\exp \langle \ln f(T)x, x \rangle}{f (\langle Tx, x\rangle)} \leq \exp [ \langle f'(T)[f(T)]^{-1}Tx, x \rangle -\langle Tx, x\rangle \langle f'(T)[f(T)^{-1}]x, x \rangle ]. 
\end{equation*}
\end{itemize}
\end{Theorem}
Substituting the log-convex differentiable function $f(t)=t^{-r}, t>0, r>0$ for part (ii) in Theorem \ref{inequalty'}, we obtain the following corollary.
\begin{Corollary}
Let $T\in \mathfrak{B}(\mathsf{H})$ be a positive selfadjoint operator with $\sigma_S(T)\subset [m, M]$ for two real numbers $0<m<M$. If $T$ is invertible, then for each unit vector $x\in\mathsf{H}$ and $r>0$, 
\begin{equation*}
\langle Tx, x \rangle^r \exp \langle\ln (T^{-r}) x,x\rangle \leq \exp [ r(\langle Tx,x\rangle \langle T^{-1}x,x\rangle -1)].
\end{equation*}
\end{Corollary}
The following result is not only a refinement but also a reverse for the multiplicative version of Jensen's inequality,  with the same proof as Theorem 3.2 in \cite{Drag}. \begin{Theorem}\label{inequality''}
Let $J$ be an interval and $f:J\rightarrow (0, +\infty)$ be a log-convex  differentiable function on $\mathring{J}$ (the interior of $J$), whose derivative $f'$ is continuous on $\mathring{J}$. If $T\in \mathfrak{B}(\mathsf{H})$ is a selfadjoint operator with $\sigma_S(T)\subseteq [m, M] \subset \mathring{J}$ for two real numbers $m<M$, then for every unit vector $x\in\mathsf{H}$
\begin{eqnarray*}
1 &\leq & \Biggl \langle \exp \Biggl [\dfrac{f'(\langle Tx,x\rangle)}{f(\langle Tx,x\rangle)}(T- \langle Tx, x\rangle I) \Bigg ]x,x \Bigg \rangle \leq \dfrac{\langle f(T)x,x\rangle}{f(\langle Tx,x \rangle)}\\
&\leq & \Biggl \langle \exp \Biggl [f'(T) [f(T)]^{-1} (T-\langle Tx, x\rangle I) \Bigg ]x,x \Bigg \rangle.
\end{eqnarray*}
\end{Theorem}
Again, considering the  log-convex differentiable function $f(t)=t^{-r}, t>0, r>0$ for  Theorem \ref{inequality''}, we obtain the following corollary.
\begin{Corollary}
Let $T\in \mathfrak{B}(\mathsf{H})$ be a positive selfadjoint operator with $\sigma_S(T)\subset [m, M]$ for two real numbers $0<m<M$. If $T$ is invertible, then for each unit vector $x\in\mathsf{H}$ and $r>0$, 
\begin{eqnarray*}
1 &\leq & \Biggl \langle \exp [r(I- \langle Tx, x \rangle^{-1}T)]x,x \Bigg \rangle \leq  \langle T^{-r} x,x\rangle \langle Tx,x \rangle ^{r} \\
&\leq &\Biggl \langle \exp [ r(I- \langle Tx,x\rangle  T^{-1})]x,x \Bigg \rangle.
\end{eqnarray*}
\end{Corollary}
The next theorem is the quaternionic version of Theorem 3.3 in \cite{Drag}, with similar proof.
\begin{Theorem}\label{inequality'''}
Let $J$ be an interval and $f:J\rightarrow (0, +\infty)$ be a log-convex  differentiable function on $\mathring{J}$ (the interior of $J$), whose derivative $f'$ is continuous on $\mathring{J}$. If $T\in \mathfrak{B}(\mathsf{H})$ is a selfadjoint operator with $\sigma_S(T)\subseteq [m, M] \subset \mathring{J}$ for two real numbers $m<M$, then 
\begin{eqnarray*}
1& \leq & \dfrac{\Biggl \langle [f(M)]^{{\frac{T-m1_{\mathsf{H}}}{M-m}}}[f(m)]^{\frac{MI-T}{M-m}}x,x \Bigg \rangle}{\langle f(T)x,x \rangle}\\
&\leq & \dfrac{\Biggl \langle f(T)\exp \Biggl [\dfrac{(MI-T)(T-mI)}{M-m}\biggl(\dfrac{f'(M)}{f(M)}-\dfrac{f'(m)}{f(m)}\bigg) \Bigg ] x, x \Bigg \rangle}{\langle f(T)x, x\rangle}\\
&\leq & \exp \Biggl [\dfrac{1}{4}(M-m)\biggl(\dfrac{f'(M)}{f(M)}-\dfrac{f'(m)}{f(m)}\bigg)\Bigg ],
\end{eqnarray*}
for every unit vector $x\in\mathsf{H}$.
\end{Theorem}
\begin{Corollary}
Let $T\in \mathfrak{B}(\mathsf{H})$ be a positive selfadjoint operator with $\sigma_S(T)\subset [m, M]$ for two real numbers $0<m<M$. If $T$ is invertible, then for each unit vector $x\in\mathsf{H}$ and $r>0$, 
\begin{eqnarray*}
1& \leq& \dfrac{  \Biggl \langle [f(M)]^{{\frac{T-mI}{M-m}}}[f(m)]^{\frac{MI-T}{M-m}}x,x \Bigg \rangle}{\langle f(T)x,x \rangle}\leq 
 \dfrac{\Biggl \langle T^{-r}\exp \Biggl [\dfrac{r(MI-T)(T-mI)}{mM}\Bigg ]x, x\Bigg \rangle}{\langle T^{-r}x, x\rangle}\\
&\leq& \exp \Biggl [\dfrac{1}{4}r\dfrac{(M-m)^2}{mM}\Bigg ].
\end{eqnarray*}
\end{Corollary}
\section{Application}
Following \cite{Drag}, we use the log-convex function  $f(t)= \left( \dfrac{1-t}{t} \right)^r, r>0$  on $(0, 1/2)$ to obtain the special cases of the inequalities in Theorems \ref{Mondlog}, \ref{Lah}, \ref{inequalty'}. With $r=1$, the Jensen's inequalty for $f$ implies the Ky Fan's inequality
\begin{equation*}
\dfrac{1-\sum _{i=1}^n p_i t_i}{\sum _{i=1}^n p_i t_i}\leq \prod_{i=1}^n \left(\dfrac{1-t_i}{t_i}\right)^{p_i},
\end{equation*} 
where $p_i>0$ and $t_i\in (0, 1/2)$ for $i\in \{1,...,n\}$ and $\sum _{i=1}^n p_i=1$. 
\begin{Proposition}
Let $T\in \mathfrak{B}(\mathsf{H})$ be a positive selfadjoint operator with $\sigma_S(T)\subset [m, M] \subset (0, 1/2)$ for two real numbers $0<m<M$. If $T$ is invertible, then for each unit vector $x\in\mathsf{H}$ and $r>0$, 
\begin{eqnarray*}
\biggl \langle \biggl((I-T)T^{-1}\bigg)^r x,x \bigg \rangle \geq \exp \biggl\langle \ln \biggl(T^{-1}(I-T)\bigg )^r x,x\bigg\rangle\geq \biggl (\langle (I-T)x, x \rangle \langle Tx,x \rangle ^{-1}\bigg )^r.
\end{eqnarray*}
\end{Proposition}
\begin{proof}
In Theorem \ref{Mondlog}, applying the log-convex function $f(t)= \left(\dfrac{1-t}{t}\right)^r, t\in (0, 1/2), r>0$  implies the result.
\end{proof}
\begin{Proposition}
Let $T\in \mathfrak{B}(\mathsf{H})$ be a positive selfadjoint operator with $\sigma_S(T)\subset [m, M] \subset (0, 1/2)$ for two real numbers $0<m<M$. If $T$ is invertible, then for each unit vector $x\in\mathsf{H}$ and $r>0$, 
\begin{eqnarray*}
\Biggl \langle ((I-T)T^{-1})^r x,x\Bigg \rangle &\leq & \Biggl \langle \Biggl [\left(\dfrac{1-m}{m}\right)^{\frac{r(MI-T)}{M-m}}\left(\dfrac{1-M}{M}\right)^{\frac{r(T- MI)}{M-m}}\Bigg ]x,x\Bigg \rangle \\ &\leq& \dfrac{M- \langle Tx,x\rangle}{M-m}\left(\dfrac{1-m}{m}\right)^r+ \dfrac{\langle Tx,x\rangle -m}{M-m}\left(\dfrac{1-M}{M}\right)^r,
\end{eqnarray*}
and
\begin{eqnarray*}
\left(\dfrac{1- \langle Tx,x \rangle}{\langle Tx,x \rangle}\right)^r &\leq & \left(\dfrac{1-m}{m}\right)^{\frac{r(M- \langle Tx,x \rangle )}{M-m}} \left(\dfrac{1-M}{M}\right)^{\frac{r( \langle Tx,x \rangle -m )}{M-m}} \\ &\leq & \Biggl \langle \Biggl[\left(\dfrac{1-m}{m}\right)^{\frac{r(MI-T)}{M-m}}\left(\dfrac{1-M}{M}\right)^{\frac{r(T- MI)}{M-m}}\Bigg ]x,x\Bigg \rangle.
\end{eqnarray*}
\end{Proposition}
\begin{proof}
In the inequalities in Theorem \ref{Lah}, using the log-convex function $f(t)= \left(\dfrac{1-t}{t}\right)^r, t\in (0, 1/2), r>0$ completes the proof.
\end{proof}
\begin{Proposition}
Let $T\in \mathfrak{B}(\mathsf{H})$ be a positive selfadjoint operator with $\sigma_S(T)\subset [m, M] \subset (0, 1/2)$ for two real numbers $0<m<M$. If $T$ is invertible, then for each unit vector $x\in\mathsf{H}$ and $r>0$, 
\begin{eqnarray*}
1 \leq  \dfrac{\exp \biggl\langle \ln \biggl((I-T)T^{-1}\bigg)^rx,x\bigg \rangle}{\biggl((1- \langle Tx,x \rangle)  \langle Tx,x \rangle^{-1}\bigg)^r}
  \leq  \exp \biggl [r \biggl (\langle Tx,x \rangle \langle T^{-1}(I-T)^{-1}x,x \rangle - \langle (I-T)^{-1}x,x \rangle \bigg )\bigg]
\end{eqnarray*}
and
\begin{eqnarray*}
1&\leq & \biggl \langle \exp \biggl [r(1-\langle Tx,x\rangle )^{-1} \biggl (I-\langle Tx,x \rangle ^{-1} T\bigg )\bigg ]x,x\bigg \rangle\\
 &\leq & \dfrac{\biggl \langle \biggl ((I-T)T^{-1}\bigg )^r x,x \bigg \rangle}{\biggl((1- \langle Tx,x \rangle)  \langle Tx,x \rangle^{-1}\bigg)^r} \leq  \biggl \langle \exp \biggl [r(I-T)^{-1}\biggl (\langle Tx,x \rangle T^{-1} - I\bigg )\bigg ]x,x \bigg \rangle.
\end{eqnarray*} 
\end{Proposition}
\begin{proof}
In Theorem \ref{inequalty'} (ii), using the log-convex function $f(t)= \left(\dfrac{1-t}{t}\right)^r, t\in (0, 1/2), r>0$  obtains the results.
\end{proof}

\section*{Disclosure statement}

No potential conflict of interest was reported by the author.


\bibliographystyle{amsplain}

\end{document}